\documentclass[12pt,a4paper]{amsart}

\usepackage[utf8]{inputenc}
\usepackage[english]{babel}

\usepackage{latexsym}
\usepackage{amsfonts}
\usepackage{amsmath}
\usepackage{amssymb}
\usepackage{graphicx}
\usepackage{color}
\usepackage{enumitem}
\usepackage{relsize}
\usepackage{fullpage}

\usepackage{currfile}

\usepackage{amsthm}

\usepackage{mathrsfs}
\usepackage[ocgcolorlinks, linkcolor=red]{hyperref}

\newcommand{\R}{\mathbb{R}}
\newcommand{\N}{\mathbb{N}}

\newcommand{\Hcurl}[1]{H(\operatorname{curl};#1)}
\newcommand{\Hocurl}[1]{H_0(\operatorname{curl};#1)}
\newcommand{\Hscurl}[2]{H^{#1}(\operatorname{curl};#2)}
\newcommand{\Hspcurl}[3]{H^{#1,#2}(\operatorname{curl};#3)}

\newcommand{\ov}[1]{\overline{#1}}
\newcommand{\eps}{\epsilon}
\newcommand{\veps}{\varepsilon}

\newcommand\supp{\operatorname{supp}}

\newcommand\p{\partial}

\newcommand\curl{\operatorname{curl}}

\newtheorem{thm}{Theorem}[section]
\newtheorem{cor}[thm]{Corollary}
\newtheorem{lem}[thm]{Lemma}
\newtheorem{prop}[thm]{Proposition}

\theoremstyle{definition}
\numberwithin{equation}{section}

\usepackage{thmtools, thm-restate}

\begin{document}

\title[]{
On quantitative Runge approximation for the time harmonic Maxwell equations
}
\date{}
\keywords{Runge approximation; Maxwell's equations; Cauchy problem; Unique continuation}

\author[]{Valter Pohjola}
\address{BCAM - Basque Center for Applied mathematics}
\email{valter.pohjola@gmail.com}

\maketitle

\begin{abstract} 
    Here we derive some results on so called quantitative Runge approximation
    in the case of the time-harmonic Maxwell equations.
    This provides a Runge approximation having more explicit quantitative information.
    We additionally derive some results on the 
    conditional stability of the Cauchy problem for the time-harmonic Maxwell equations.    
\end{abstract}

\tableofcontents

\section{Introduction} \label{sec_intro}

\noindent
The time-harmonic Maxwell equations are given by the system
\begin{align} \label{eq_MEsys}
\left\{
\begin{aligned}
\nabla \times E -  i\omega \mu H &= 0, & & \textnormal{in} \,\Omega,\\
\nabla \times H +  i\omega \veps E &= 0, & & \textnormal{in} \,\Omega.\\
\end{aligned}
\right.
\end{align}
Here $\Omega \subset \R^3$ is a bounded domain.
In \eqref{eq_MEsys} $E$ and $H$ are vector fields corresponding to the electric and magnetic fields.
Furthermore $\veps, \mu \in W^{1,\infty}(\Omega;\R^{3\times 3})$  
are symmetric matrices with Lipschitz continuous coefficients that model an anisotropic permittivity and permeability
of a medium
(for the notations on the function spaces used here and elsewhere in this paper 
see the appendix, section \ref{sec_fspaces}). 

The system \eqref{eq_MEsys} is obtained by considering time-harmonic fields of the 
form  $E'(x,t) = E(x)e^{i\omega t}$ and $H'(x,t) = H(x)e^{i\omega t}$ in the time dependent
form of Maxwell's equations and factoring out the time dependence.
The time independent system  \eqref{eq_MEsys} provides a basic tool 
for understanding electromagnetic wave phenomena. For more details on this see for instance \cite{J}. 
Here we assume that the angular frequency $\omega$ is such that the system \eqref{eq_MEsys}
in conjunction with the boundary condition $(\nu \times E) |_{\p\Omega} = f \in TH_D^{-1/2}(\p\Omega)$,
(see section \ref{sec_fspaces})
is well-posed. For more on the well-posedness see section 2 of \cite{HLL}. 

\bigskip
\noindent
The main aim of this paper is to study quantitative Runge approximation for the
system \eqref{eq_MEsys}.
The term Runge approximation refers  to the approximation of a solutions of 
an partial differential equation (PDE) on a given set, with solutions on a larger set  containing the given set.
The Runge Theorem in complex analysis, is an instance of this type of an approximation property.
Similar ways of approximating are also 
possible in the context of (other) PDEs, and these ideas were first considered in \cite{L} and \cite{M},
and more generally by \cite{B2}.
Runge approximation has become a increasingly used tool when dealing with 
diverse range of inverse problems.
For some recent applications in different contexts see e.g. \cite{HPS,RS2,LLS,DKSU}.
The quantitative version of Runge approximation, which is the focus of this paper,
was first considered in \cite{RS}, where it was studied for Schr\"odinger type equations. 
It should also be noted that quantitative Runge approximation is
closely connected to work done in the context of control theory \cite{R}.

The Runge approximation problem we consider here can be described as follows.
It will be convenient to abbreviate the above system \eqref{eq_MEsys} by the notation
$$
\mathcal{M}(E,H) = 0, \text{ in }\Omega.
$$
We will further make the assumption that
\begin{align} \label{eq_ellip}	
\left\{
\begin{aligned}
&0 \leq c|\xi|^2  \leq  \mu(x) \,\xi \cdot \xi  < c^{-1}|\xi|^2, \\
&0 \leq c|\xi|^2  \leq  \veps(x) \,\xi \cdot \xi  < c^{-1}|\xi|^2, \\
\end{aligned}
, \quad \text{ for } x \in \Omega \text{ and } \xi \in \R^3,
\right.
\end{align}
and  where $c >0$.
We will let $\Omega \subset \R^3$ be a bounded domain with  Lipschitz boundary. First
we define the set of solutions with boundary data supported on an open set 
$\Gamma \subset \p\Omega$, i.e. let
$$
\mathcal{S}_\Omega :=\big\{ (E,H) \in \Hcurl{\Omega}^2 \,:\, \mathcal{M}(E,H)=0,
\, \supp(\nu\times E|_{\p\Omega})\subset \Gamma\big\}. 
$$
For the notations used here see section \ref{sec_fspaces}.
Notice that $\Gamma \subset \p \Omega$ might be an arbitrarily  small part
of $\p \Omega$.
Next we define  $\mathcal{S}_A$, where $A \subset \subset \Omega$ has a Lipschitz boundary,
and $\Omega \setminus \ov{A}$ is connected. The set $\mathcal{S}_A$ is given by
\begin{align}  \label{eq_SA}
\mathcal{S}_A :=\big \{ (E,H) \in \Hcurl{A}^2 \;:\; \mathcal{M}(E,H)=0, 
\text{ in } A \big \}. 
\end{align}
The Runge approximation problem is then to show that we can approximate fields in $\mathcal{S}_A$,
by the fields $\big( E, H\big)|_A$, where $(E,H) \in S_\Omega$.
The main result is  the following quantitative Runge approximation property, 
which is proved in section \ref{sec_qrunge}.

\begin{restatable}{thm}{thmOne} \label{thm_qrunge}
Given $(E, H) \in \mathcal{S}_A \cap L^{p}(A)^6$, with $p>2$, there are constants
$m > 0$, $C > 0$, and $(E_j,H_j) \in \mathcal{S}_\Omega$, $j \in \N$, such that
$$
\big\|(E,H) - (E_j ,H_j)|_{A} \big\|_{\Hspcurl{0}{2}{A}^2} 
\leq \frac{C}{j} \big\| (E, H)\big\|_{\Hspcurl{0}{q}{A}^2},
$$
for $q \in (2, p\,]$, and for which
$$
\| (E_j,  H_j ) \|_{\Hcurl{\Omega}^2}  
\leq C  e^{Cj^{2/m}} \| (E, H)  \|_{L^2(A)^6}. 
$$
\end{restatable}

\noindent
The main novelty of this result is that it extends the quantitative Runge approximation
result in \cite{RS} to  the Maxwell system \eqref{eq_MEsys}. This result is also
connected to the study of localization of solutions to \eqref{eq_MEsys} in the fixed frequency case.
The localization of solutions play an important part in certain reconstruction methods in 
inverse problems, and in particular for the so called monotonicity method (see e.g. \cite{G} and \cite{HPS}).
In \cite{HLL}, the authors studied the localization
of solutions to \eqref{eq_MEsys}, and gave  a way to construct a sequence of solutions
whose $L^2$-norm increases in a set $M \subset \Omega$ and decreases in a set $D \subset \subset \Omega$
(see Theorem 3.1 in \cite{HLL}). 
Theorem \ref{thm_qrunge} can be used to construct localized
solutions, of the kind in \cite{HLL}, with an upper bound on how badly 
the localized solution can behave outside the set of localization.

In proving Theorem \ref{thm_qrunge} we adapt the arguments in \cite{RS} to the Maxwell case.
Let us mention some of the particularities in dealing with the problem in case of the time harmonic
Maxwell equations.
Note firstly that Theorem \ref{thm_qrunge} applies to elements in $\mathcal{S}_A$
that also have some additional integrability, i.e. elements that are also in $L^p(\Omega)^6,$ for some $ p>2$. 
By the qualitative  Runge argument in \cite{HLL} we know that all elements of $\mathcal{S}_A$
can be estimated by elements in $\mathcal{S}_\Omega$.
To what extent it is possible to obtain a quantitative approximation without additional
integrability of regularity assumptions on the element in $\mathcal{S}_A$ that is approximated,
remains open. Note also that in the scalar case in \cite{RS} (see Theorem 1.2),
the solutions are in contrast in $H^1(\Omega)$ by elliptic regularity, and this type of assumption
is thus automatically satisfied by the elements in the analogue of the set $\mathcal{S}_A$
studied in \cite{RS}.

Another feature 
in proving Theorem \ref{thm_qrunge} in the Maxwell setting, is that we need to use $L^p$-spaces, with $p\neq2$.
This stems from that the conditional stability result for the Cauchy problem (see Theorem \ref{thm_CauchyProb_1})
that we apply, applies to boundary data in $TB^{-1/p}_{D}(\p\Omega)$ with $p>2$.
For much of this we are however able to apply the $L^p$-results in \cite{KS} and \cite{M2}.
In addition we can use interpolation, when applying Theorem \ref{thm_CauchyProb_1}.

Note also that the fact that we use 
boundary data in $TH_D^{-1/2}(\p\Omega)$ results in a solution $(E,H) \in \Hcurl{\Omega}^2$.
The space $\Hcurl{\Omega}$ is not compactly embedded in $L^2(\Omega)^3$. The compactness is needed
in Lemma \ref{lem_SVD}, where we construct a singular value decomposition. Interior elliptic regularity
can be however used to show that $(E,H) \in H^1_{loc}(\Omega)^6$, when $\veps$ and $\mu$ are Lipschitz,
which can be used to obtain the compact embedding.

\bigskip
\noindent
In proving Theorem \ref{thm_qrunge} we need the conditional stability of the Cauchy problem
for \eqref{eq_MEsys}. This is a result of independent interest and is stated as Theorem \ref{thm_CauchyProb_1}.
From Theorem \ref{thm_CauchyProb_1} one derives a quantitative version
of the unique continuation principle, which is in turn the key estimate in deriving the  Runge 
approximation result of Theorem \ref{thm_qrunge}.

The Cauchy data problem that we  are concerned with is 
\begin{align} \label{eq_CauchyProb}
\left\{
\begin{aligned}
\mathcal{M}(E,H) &= 0, & & \textnormal{in} \,\Omega,\\
\nu \times E &= f, & & \textnormal{on} \, \Gamma, \\
\nu \times H &= g, & & \textnormal{on} \, \Gamma,
\end{aligned}
\right.
\end{align}
and where $\Gamma \subset \p \Omega$ an open set. The main stability result is stated below, 
and is proved in sub section \ref{sec_thmCauchy}.

\begin{restatable}{thm}{thmTwo} \label{thm_CauchyProb_1}
Suppose that $(E, H) \in L^p(\Omega)^6$, with $p>2$ solve the Cauchy problem
\eqref{eq_CauchyProb}, and  that
$$
\eta \geq \| f \|_{TB_{D}^{-1/p}(\Gamma)} +  \| g \|_{TB_{D}^{-1/p}(\Gamma)}, 
\quad  \zeta \geq \| E \|_{L^p(\Omega)^3} + \| H \|_{L^p(\Omega)^3},
$$
Then we have the estimate
\begin{align} \label{eq_CauchyEst}
\big\|(E,H) \big\|_{\Hcurl{\Omega}^2} 
\leq C (\zeta + \eta) \Bigg( \log \frac{\zeta + \eta}{ \eta} \Bigg)^{-m},
\end{align}
where $C,m > 0$ are positive constants. 
\end{restatable}

\noindent
The conditional stability of the Cauchy problem has been investigated extensively in inverse problems
and the study of ill-posed problems. 
The Cauchy problem for elliptic PDE, was in fact one of the first problems 
that suggested how to deal with ill-posedness. The solution $(E,H)$ of a PDE as \eqref{eq_CauchyProb},
does not in general depend continuously on $f$ and $g$, as was observed by Hadamard in \cite{Ha}.
The stability problem can in this sense be
ill-posed (for a more detailed discussion of this phenomena and Hadamard's example
see section 1.1. in \cite{ARRV}). One can nevertheless often prove so called conditional stability
results, which are of relevance in many situations.
A conditional stability result, such as Theorem \ref{thm_CauchyProb_1},
uses an apriori bound on the size of the solutions in the estimate (which in Theorem
\ref{thm_CauchyProb_1} is the assumption on $\zeta$).
For more background on the Cauchy problem in the context 
of inverse problems see \cite{ARRV} and \cite{Ch}. 

Theorem \ref{thm_CauchyProb_1} extends the type of global stability 
result for the Cauchy problem  found in \cite{ARRV} to the case of Maxwell's equations
(see Theorem 1.9 in \cite{ARRV}).
This type of problem has in the Maxwell case, been studied to some degree. Earlier results 
on the conditional stability of the Cauchy problem for Maxwell's equations
can be found in \cite{EY}. Note also  that the Cauchy problem has been studied
extensively (see e.g. \cite{ARRV} and \cite{C} and corresponding references).

The proof of Theorem \ref{thm_CauchyProb_1} follows the methods of \cite{ARRV} fairly closely.
The starting point are the three ball inequalities given in \cite{NW}. From this one derives
a local and global propagation of smallness results, which in turn yields the stability estimate.
Here we use the fact the Maxwell system in \eqref{eq_MEsys}, can be analyzed
using the second order system
$$
\nabla \times (\mu^{-1}\nabla \times E) - \omega^2\eps E = 0.
$$
The above equation is connected to a scalar second order elliptic equation. Note 
that for $\mu$ and $\eps$ that are constant scalars it reduces to three copies of the Helmholtz equation.
One difference that becomes important when proving Theorem \ref{thm_CauchyProb_1},
is that $TH_D^{-1/2}(\p\Omega)$ boundary data gives
solutions to equation \eqref{eq_MEsys} that are in $\Hcurl{\Omega}$, and we cannot use
Sobolev embedding as is done in \cite{ARRV}, when proving the global propagation of smallness
in Proposition \ref{prop_globProp}. 
We deal with this by considering boundary data in the space $TB^{-1/p}_{D}(\p\Omega)$, 
where $p>2$, (see section \ref{sec_fspaces} for the definition) and using the Hölder inequality instead.

Some further remarks on Theorem \ref{thm_CauchyProb_1} are that the methods here could also be used to prove
a version of the theorem with source terms, and that an interior stability estimate as 
Theorem 1.7 in \cite{ARRV} could also be proved similarly as Theorem \ref{thm_CauchyProb_1}.

\medskip
\noindent
This paper is organized as follows. We begin by proving Theorem \ref{thm_CauchyProb_1} in section
\ref{sec_Cauchy}. For this 
we need a three ball inequality that follows directly from the results in \cite{NW} which
are in section \ref{sec_3balls}. In subsections \ref{sec_intprop} and \ref{sec_globprop} we 
prove the interior and global propagation of smallness. In subsection \ref{sec_thmCauchy}
we give the proof of Theorem \ref{thm_CauchyProb_1}. The  next part of the paper containing the Runge 
approximation argument starts in section \ref{sec_Runge_approximation}. First we derive
a quantitative form of a unique continuation principle in section \ref{sec_qucp}. In the next 
subsection \ref{sec_qrunge} we prove the main Theorem \ref{thm_qrunge}. In the appendix in section 
\ref{sec_fspaces} we review definitions and facts relating to the functions spaces used
through out the paper.

\section{Stability of the Cauchy problem} \label{sec_Cauchy}

\noindent
In the following subsections we will prove Theorem \ref{thm_CauchyProb_1}.
For the whole of section \ref{sec_Cauchy} we will assume that $\Omega \subset \R^3$ is a bounded Lipschitz domain.
We will let $L_{\eps,\mu}$ denote the operator
$$
L_{\veps,\mu} E := \nabla \times (\mu^{-1}\nabla \times E) - \omega^2\eps E.
$$
Note that setting $H = \tfrac{1}{i\omega }\mu^{-1}\nabla \times E$ using \eqref{eq_MEsys},
yields the equation
$$
L_{\veps,\mu} E = 0, \quad \text{in } \Omega.
$$
It will be convenient at times to use this second order elliptic equation, in places of \eqref{eq_MEsys}.

\subsection{Three ball inequalities} \label{sec_3balls}

Our starting point is the three ball inequalities in \cite{NW}.
We assume that $M >0$ is such that
\begin{align} \label{eq_M}
\|\veps\|_{W^{1,\infty}(\Omega;\R^{3\times 3})}, \;
\|\mu\|_{W^{1,\infty}(\Omega;\R^{3\times 3})} \leq M.
\end{align}
Assume furthermore that $E \in L^2(\Omega)^3$ solves
\begin{align} \label{eq_Eeq}
L_{\veps,\mu} E
&= 0, \quad \textnormal{ in }\Omega,
\end{align}

\noindent
The following result is obtained in \cite{NW}, see Theorem 1.1.

\begin{prop} \label{prop_3balls}
There exists $\rho > 0$ and  $\tau \in (0,1)$ such that
for all $E$ that are  solutions to  \eqref{eq_Eeq}, and
$r_1< r_2 < r_3/2 < \rho$, with $B(x,r_3) \subset \Omega$,
we have the estimate 
\begin{align}
\|E\|_{\Hcurl{B_2}}
\leq C
\|E\|^{\tau}_{\Hcurl{B_1}}
\|E\|^{1-\tau}_{\Hcurl{B_3}},
\end{align}
where $C$ depends on $M,\tau,r_2,r_3$ and $c$ from \eqref{eq_ellip}, and $B_j:=B(x,r_j)$.
\end{prop}

\noindent
For a solution $E \in L^2(\Omega)^3$ to the inhomogeneous equation
\begin{align} \label{eq_Eeq_inh}
L_{\veps,\mu} E &= F + \nabla \times \tilde F \;\in \Hscurl{-1}{\Omega}, \quad \textnormal{in} \,\Omega,
\end{align}
where $F,\tilde F \in L^2(\Omega)^3$ and 
$$
\|F\|_{L^2(\Omega)^3} + \|\tilde F\|_{L^2(\Omega)^3}  \leq M_0,
$$
we have the following result. Note that the formulation of the estimate 
of the corollary, is such
that it is in a convenient for later use.

\begin{cor}  \label{cor_3balls_inh}
There exists $\rho > 0$  and $\tau \in (0,1)$ such that
for all $E$ that are  solutions to  \eqref{eq_Eeq_inh}, and
$r_1< r_2 < r_3/2 < \rho$ with $B(x,r_3) \subset \Omega$, we have the estimate 
\begin{align*}
\big(\|E\|_{\Hcurl{B_2}} &+ M_0\big) 
\leq C
\big(\|E\|_{\Hcurl{B_1}} + M_0\big)^{\tau}
\big(\|E\|_{\Hcurl{B_3}} + M_0\big)^{1-\tau},
\end{align*}
where $C$ depends on $M,\tau,r_2,r_3$ and $c$ from \eqref{eq_ellip}.
\end{cor}

\begin{proof}
Let $E$ be a solution to \eqref{eq_Eeq_inh} and
consider the solution $U \in \Hcurl{\Omega}$ to the boundary value problem
\begin{align*} 
\left\{
\begin{aligned}
\nabla \times \big( \mu^{-1} \nabla \times U \big) 
- \omega^2 \veps U &= F + \nabla \times \tilde F , & & \textnormal{in} \,\Omega, \\
\nu \times U &= 0, & & \textnormal{on} \, \partial \Omega.
\end{aligned}
\right.
\end{align*}
The function $E-U$ solves the homogeneous equation \eqref{eq_Eeq}.
Using the inequality of Proposition \ref{prop_3balls} and 
the estimates for the solution $U$ in terms of the source terms
(see  Theorem 2.1 in \cite{HLL}), we have that 
\begin{align*}
\|E-U\|_{\Hcurl{B_2}}
&\leq C
\|E-U\|^{\tau}_{\Hcurl{B_1}}
\|E-U\|^{1-\tau}_{\Hcurl{B_3}} \\
&\leq C
(\|E\|_{\Hcurl{B_1}} + M_0)^{\tau}
(\|E\|_{\Hcurl{B_3}} + M_0)^{1-\tau}.
\end{align*}
On the other hand we have that
$$
\|E\|_{\Hcurl{B_2}} \leq
\|E-U\|_{\Hcurl{B_2}} + M_0.
$$
By using the trivial estimate that
$$
M_0\leq 
(\|E\|_{\Hcurl{B_1}} + M_0)^{\tau}
(\|E\|_{\Hcurl{B_3}} + M_0)^{1-\tau},
$$
it is straight forward to see that the above estimates 
imply the claim.

\end{proof}

\subsection{Interior propagation of smallness} \label{sec_intprop}

\noindent The three ball inequalities of the previous section 
can be extended to  give an estimate of the norm on an open set $G \subset \subset \Omega$.
We will consider a solution $E \in \Hcurl{\Omega}$ to the equation
\begin{align} \label{eq_intEeq}
L_{\veps,\mu} E &= F + \nabla \times \tilde F \in \Hscurl{-1}{\Omega}, \quad \textnormal{in } \Omega.
\end{align}
The procedure of extending the three ball inequalities 
to obtain an estimate for $G$, is called propagation of smallness.
We have the result below, which we prove following the methods of 
of Proposition 5.1 in \cite{ARRV}.
\begin{prop} \label{eq_propIntProp}
Let $B(x_0,r_0) \subset \Omega$, and assume that
$G \subset \Omega$ is an open and connected set such that 
$$
\operatorname{dist} ( G, \p \Omega) > h, \quad B(x_0, \tfrac{r_0}{2}) \subset G.
$$
where $h \leq \{ 2 \rho, r_0/2 \}$, and $\rho$ is as in Corollary \ref{cor_3balls_inh}.
Assume furthermore that $E$ solves \eqref{eq_intEeq} and that $\eta$ and $\zeta$ are such that
$$
\| E \|_{\Hcurl{B(x_0,r_0)}} \leq \eta, \quad \|E\|_{\Hcurl{\Omega}} \leq \zeta,
$$
then we have that
\begin{align} \label{eq_intProp}
\|E\|_{\Hcurl{G}} \leq C
\big(\eta + M_0 \big)^{\delta}
\big(\zeta + M_0\big)^{1-\delta},
\end{align}
where the constants are such that $\delta \in (0,1)$, and $C$ depends on $h$.
\end{prop}

\begin{proof}
We begin by fixing three radii $r_1 < r_2 < r_3$, as
$$
r_3 := \tfrac{h}{2}, \; r_2 := \tfrac{r_3}{3}, \; r_1 := \tfrac{r_2}{3}.
$$
The strategy is to cover $G$ with balls on which we can use the earlier 
three ball inequality of Corollary \ref{cor_3balls_inh}.
To this end we consider the larger set $G^{r_1} \supset G$ given by
$$
G^{r_1} := \{ x \in \Omega \; : \; \operatorname{dist}(x,G) < r_1 \}.
$$
Let $y \in G^{r_1}$ and $\gamma : [0,1] \to G^{r_1}$ be a path such that
$
\gamma(0) = x_0 $ and $ \quad \gamma(1) = y.
$
We define $0=t_0 < \dots < t_N = 1$ by
$$
t_{k+1} := \max \{ \,t \;:\; |\gamma(t) - x_k| = 2r_1 \}, \quad k = 2,..,N-2,
$$
where $x_k := \gamma(t_k)$. It is easy to check that
\begin{enumerate}
\item $B(x_k,r_1)$, $k=1,..,N$ are disjoint.

\item $B(x_{k+1},r_1) \subset B(x_k, r_2)$, for $k=1,..,N-1$.
\end{enumerate}
The second condition allows us to apply 
the estimate of Corollary \ref{cor_3balls_inh}, to obtain that
\begin{align*}
\|E\|_{\Hcurl{B(x_{k+1},r_1)}} + M_0
\leq C
\big(\|E\|_{\Hcurl{B(x_{k},r_1)}} + M_0\big)^{\tau}
\big(\zeta + M_0\big)^{1-\tau}.
\end{align*}
We write this as 
\begin{align*}
m_{k+1} := \frac{\|E\|_{\Hcurl{B(x_{k+1},r_1)}} + M_0}
{\zeta + M_0}
\leq 
C\Big( \frac{\|E\|_{\Hcurl{B(x_k,r_1)}} + M_0}
{\zeta + M_0}
\Big)^\tau.
\end{align*}
We thus have that $m_{k+1} \leq C m_k^\tau$. This implies that
$$
m_N \leq C^{1+\tau+\tau^2+...+\tau^N}m^{\tau^N}_0.
$$
Using this and the fact that $r_0 > r_1$,
we obtain an estimate in the vicinity of the point $y$
of the form that we are seeking, i.e.
\begin{align} \label{eq_estAty}
\|E\|_{\Hcurl{B(y,r_1)}} 
\leq C'
\big(\|E\|_{\Hcurl{B(x_0,r_0)}} + M_0\big)^{\delta}
\big(\zeta + M_0\big)^{1-\delta},
\end{align}
where we set 
$$
C' := C^{\frac{1}{1-\tau}} \geq  C^{1+\tau+\tau^2+...+\tau^N}
\quad \text{and} \quad
\delta := \tau^{ C_2|\Omega|/h^n} \geq \tau^N.
$$
Notice that the last estimate is possible for some $C_2$ independent of $N$,
since we chose the balls $B(x_k,r_1)$ to be disjoint,
and therefore we can estimate $N$, by
$$
N \leq |\Omega| / (r^n_1 C) \leq C_2|\Omega|/h^n.
$$
The above choice of $C'$ and $\delta$ are thus such that \eqref{eq_estAty} 
does not depend on $N$.

The last step is to cover $G$ with appropriate sets and then
use \eqref{eq_estAty} to obtain an estimate pertaining to $G$.
To this end we tessellate
$\R^n$ with cubes $K_j$, with side length $l=2r_1 / \sqrt{n}$. Let
$
\{ K_j \;:\; j\in J \},
$
be the ones that intersect $G^{r_1}$. Note that $l$ is such that for every
$K_j$, $j \in J$ there is $y_j \in G^{r_1}$, such that $K_j \subset B(y_j, r_1)$.
For $J$ we have the estimate
$$
J \leq \frac{|\Omega|}{|K_j|} = \frac{|\Omega|}{2^n r_1^n} n^{n/2}.
$$
We can now use \eqref{eq_estAty} to obtain the desired estimate. We have that
\begin{align*}
\| E \|^2_{\Hcurl{G}} &\leq \sum_{j \in J} \| E \|^2_{\Hcurl{K_j}}
\leq \sum_{j \in J} \|  E \|^{2}_{\Hcurl{ B(y_j, r_1) }} \\
&\leq
J (C')^2  (\| E\|_{\Hcurl{B(x_0,r_1)}} + M_0)^{2\delta} (\zeta + M_0)^{2-2\delta},
\end{align*}
which proves the claim.
\end{proof}

\subsection{Global propagation of smallness} \label{sec_globprop}

\noindent
In this section we derive a global version of the estimate in \ref{eq_propIntProp} 
in the previous section.
Our aim is more specifically to obtain an estimate where the set $G$ in Proposition
\ref{eq_propIntProp} is replaced by $\Omega$. We now assume 
that $E \in \Hspcurl{0}{p}{\Omega}$, for some $p>2$, and solves the equation
\begin{align} \label{eq_intEeq2}
L_{\veps,\mu} E &= F + \nabla \times \tilde F \in \Hscurl{-1}{\Omega}, \quad \textnormal{in } \Omega,
\end{align}
where $F,\tilde F \in L^2(\Omega)^3$.
We have the following proposition.

\begin{prop} \label{prop_globProp}
Assume that $E \in \Hspcurl{0}{p}{\Omega}$, for some $p>2$, solves \eqref{eq_intEeq2}, and let $B(x_0,r_0) \subset \Omega$.
Assume further  $\eta$ and $\zeta$ are, such that
\begin{align*}
\| E \|_{\Hcurl{B(x_0,r_0)}} \leq \eta, \quad
\| E \|_{\Hspcurl{0}{p}{\Omega}} \leq \zeta, 
\end{align*}
There exist $C>0$ and $m>0$ such that
\begin{align*}
\|E\|_{\Hcurl{\Omega}} \leq (\zeta + M_0) \omega\Big( \frac{\eta + M_0}{\zeta + M_0} \Big),
\end{align*}
where
$
\omega(t) := C \big( \log ( 1/t) \big)^{-m},\; t \in (0,1), 
$
and $m >0$.
\end{prop}

\begin{proof} We begin by defining the subset $\Omega_r \subset \Omega$, as
$$
\Omega_r := \{ x \in \Omega \;:\; \operatorname{dist}(x, \p \Omega) > r \}.
$$
By following the same line of reasoning as in the first part of the proof of 
Theorem 5.3 in \cite{ARRV}, one can shown that
for small enough $r>0$, we have the following estimate
\begin{align} \label{eq_Omegar}
\| E \|_{\Hcurl{\Omega_r}} \leq C_1 (\zeta + M_0) r^{-n/2}
\Big(  \frac{\eta + M_0}{\zeta + M_0} \Big)^{C_2 r^D}.
\end{align}
See formula (5.38) in \cite{ARRV}. The above estimate resembles \eqref{eq_intProp},
note however that in the above estimate there is an explicit dependence on $r>0$.
The proof of \eqref{eq_Omegar} in \cite{ARRV} is based on 
the estimate corresponding to \eqref{eq_intProp} 
and the argument depends otherwise only on the geometry of the domain $\Omega$.
This proof
works thus in the same way in the Maxwell case, once one has the estimate \eqref{eq_intProp}.
For the details see the proof of Theorem 5.3 in \cite{ARRV}. 

\medskip
To obtain an estimate that applies to $\Omega$ we need an estimate that applies to
$\Omega \setminus \Omega_r$. 
By the H\"older inequality we have for the  $p>2$ in the statement of the claim, that
\begin{align*}
\| E \|_{L^2(\Omega \setminus \Omega_r)} 
&\leq |\Omega \setminus \Omega_r|^{\frac{1}{2}-\frac{1}{p}} 
\| E \|_{L^p(\Omega \setminus \Omega_r)} \\
&\leq C r^{\frac{1}{2}-\frac{1}{p}} 
\zeta.
\end{align*}
Similarly using the assumption on $\zeta$ we have that
\begin{align*}
\| \nabla \times E \|_{L^2(\Omega \setminus \Omega_r)} 
&\leq
C r^{\frac{1}{2}-\frac{1}{p}} \zeta.
\end{align*}
Combining these two estimates with \eqref{eq_Omegar} gives that
\begin{align*}
\| E \|_{\Hcurl{\Omega}} \leq C_1 (\zeta + M_0) \Big[ r^{-n/2}
\Big(  \frac{\eta + M_0}{\zeta + M_0} \Big)^{C_2 r^D}
+
r^{\frac{1}{2}-\frac{1}{p}} \Big].
\end{align*}
Let us simplify the expression by setting 
$$
\tau := r^D, \; \theta := \frac{1}{D}\Big(\frac{1}{2}-\frac{1}{p} \Big), 
\, \kappa := \Big( \frac{\eta + M_0}{ \zeta + M_0} \Big)^{C_2},
\sigma := \frac{n}{2D}.
$$
The above estimate is then
\begin{align*}
\| E \|_{\Hcurl{\Omega}} \leq C  (\zeta + M_0)( \tau^{-\sigma}\kappa^\tau + \tau^\theta).
\end{align*}
Consider the expression
\begin{align*}
\tau^{-\sigma}\kappa^\tau + \tau^\theta 
=
\tau^{-\sigma}e^{\tau \log ( \kappa)} + \tau^\theta. 
\end{align*}
Notice firstly that $\kappa \in [0,1]$. We eliminate the exponential, by using the estimate
$e^{-s} < 1/s$, when $s>0$, i.e.
\begin{align*}
\tau^{-\sigma}\kappa^\tau + \tau^\theta 
\leq
\tau^{-\sigma}\frac{1}{\tau \log ( 1/\kappa)} + \tau^\theta. 
\end{align*}
To make the terms on the right hand side equal we choose
$$
\tau := \Big( \frac{1}{ \log (1 / \kappa)} \Big)^{\frac{1}{1+\theta+\sigma}}.
$$
Hence we have the estimate
\begin{align*}
\| E \|_{\Hcurl{\Omega}} \leq C 
(\zeta + M_0) \log\Big( \frac{\zeta + M_0}{\eta + M_0} \Big)^{ \frac{-\theta}{1+\theta+\sigma}}.
\end{align*}

\end{proof}

\subsection{Proof of Theorem \ref{thm_CauchyProb_1}} \label{sec_thmCauchy}
In this subsection we prove the stability estimate of Theorem \ref{thm_CauchyProb_1}.
Recall that the Cauchy problem of Theorem \ref{thm_CauchyProb_1} considers solutions 
$E,H \in L^p(\Omega)^3$, for some $p>2$, 
to the problem
\begin{align} \label{eq_CauchyProbSys2}
\left\{
\begin{aligned}
\nabla \times E -  i\omega \mu H &= 0, & & \textnormal{in} \,\Omega,\\
\nabla \times H +  i\omega \veps E &= 0, & & \textnormal{in} \,\Omega,\\
\nu \times E &= f, & & \textnormal{on} \, \Gamma, \\
\nu \times H &= g, & & \textnormal{on} \, \Gamma,
\end{aligned}
\right.
\end{align}
where $f,g \in TB_{D}^{-1/p}(\Gamma)$, and $\Gamma \subset \p\Omega$ is open. We moreover
assumed that $\p \Omega$ is Lipschitz.

\medskip
\noindent
We will derive the stability estimate for the Cauchy problem from the estimates 
on the source problem derived in the previous sections. We will hence transform
the Cauchy problem, to a source problem in a larger domain.

Assume that $p \in \Gamma$, and that $D' \subset \R^3$ is a cylinder, so that
in a suitable coordinate system $p=0$ and
$$
D' := \{ (x',x_3) \in \R^3 \,:\, |x'| \leq \rho_0, \, |x_3|\leq L_0 \}.
$$
The cylinder $D'$ can be chosen, so that $\p\Omega \cap D'$ is the graph of
a Lipschitz function defined on $\{x' \in \R^2\,:\, |x'| \leq \rho_0\}$.
We can furthermore arrange things, so that $\p (\Omega \cup D')$ is Lipschitz regular.
And finally we choose $D'$ so that
$$
D' \cap \p \Omega \subset \subset \Gamma.
$$
We also use the notation
$$
D:= D' \setminus \Omega,  \quad\quad
\Gamma' := \Gamma \cap D. 
$$
Now extend the  permittivity $\veps$ and permeability $\mu$, as $\veps_e$ and 
$\mu_e$, so that
$$
\veps_e, \mu_e \in W^{1,\infty}(\Omega \cup D; \R^{3\times 3}), \quad \veps_e, 
\mu_e > \tilde \gamma_0 >0.
$$
The next Lemma gives us extensions of $E$ and $H$ to the set $\Omega \cup D$.

\begin{lem} \label{lem_Ee}
There exists
$$
E_e, H_e \in \Hspcurl{0}{p}{\Omega\cup D},
$$
for some $p>2$, such that
$$
E_e|_\Omega = E,\quad 
H_e|_\Omega = H, 
$$
and which obey the estimate 
$$
\| H_e \|_{ \Hspcurl{0}{p}{ D}  }
+ \| E_e \|_{ \Hspcurl{0}{p}{ D}  }
\leq C\big( 
\| f \|_{ TB^{-1/p}_{D}(\Gamma)}
+ \| g \|_{ TB^{-1/p}_{D}(\Gamma)}
\big).
$$
\end{lem}

\begin{proof}
We prove the claim for $E$. The assumption is such that $E \in \Hspcurl{0}{p}{\Omega}$. 
By Proposition 2.4 in \cite{M2} we have that 
$(\nu \times E)|_{\Gamma'} = f \in TB^{-1/p}_{D}(\Gamma')$.

Lemma \ref{lem_trace} gives a continuous  right inverse 
$\eta_{t,p} : TB^{-1/p}_{D}(\Gamma') \to \Hspcurl{0}{p}{D}$ 
to the tangential  operator $(\nu \times \, \cdot)|_{\Gamma'}$, and we set
$W := \eta_{t,p} f \in \Hspcurl{0}{p}{D}$, so that
$$
\| W \|_{ \Hspcurl{0}{p}{D} } \leq C  \| f \|_{ TB^{-1/p}_{D}(\Gamma)}.
$$
We now set
$$
E_e(x) =:\left\{
\begin{aligned}
&E(x), & x \text{ in } \Omega, \\
&W(x), & x \text{ in } D. \\
\end{aligned}
\right.
$$
It is straight forward to check that $E_e \in \Hspcurl{0}{p}{\Omega \cup D}$,
using the fact that the tangential traces of $E$ and $W$ are in agreement on $\Gamma'$.

\end{proof}

\noindent
The next Lemma shows that the extensions $E_e$ and $H_e$ solve a source problem 
in the extended domain $\Omega \cup D$. 

\begin{lem} \label{lem_Eesol}
There exists $\tilde F, \tilde G \in L^p(\Omega \cup D)^3$, such that 
$(E_e,H_e)$ solves 
\begin{align} \label{eq_extProbSys}
\left\{
\begin{aligned}
\nabla \times E_e -  i\omega \mu H_e &= \tilde F, & & \\
\nabla \times H_e+  i\omega \veps E_e &= \tilde G, & &\\
\end{aligned}
\right. \quad \textnormal{in} \,\Omega\cup D.
\end{align}
Moreover we  have  the estimate
$$
\| \tilde F\|_{L^p(\Omega \cup D)} 
+ \| \tilde G\|_{L^p(\Omega \cup D)} 
\leq C\big( 
\| f \|_{ TB^{-1/p}_{D}(\Gamma)}
+ \| g \|_{ TB^{-1/p}_{D}(\Gamma)}
\big).
$$
\end{lem}

\begin{proof}
We write the system \eqref{eq_extProbSys} as the corresponding 2nd-order equation, and
consider its weak solutions (for details see e.g. Theorem 2.1 in \cite{HLL}). 
This amounts to finding  $\tilde F, \tilde G \in L^2(\Omega \cup D)$
such that
$$
B_{\Omega \cup D}(E_e,U) = \int_{\Omega \cup D} i \omega \tilde G \cdot \ov{U} 
- \mu^{-1}_e \tilde F \cdot \nabla \times \ov{U}\,dx, \quad \forall U \in \Hocurl{\Omega\cup D}
$$
and where $B$ is the sesquilinear form
$$
B_S(E_e,U) := \int_{S} \mu^{-1}_e \nabla \times E_e  \cdot \nabla \times  \ov{U} 
- \omega^2 \veps_e E_e \cdot \ov{U} \,dx.
$$
On the set $\Omega$, we have by using the weak definition of boundary condition 
$(\nu \times H )|_\Gamma = g \in TB_{D}^{-1/p}(\Gamma) \subset TH^{-1/2}(\Gamma)$, that
\begin{align*} 
B_\Omega(E_e,U) = 
\langle i \omega \nu \times H,\, U  \rangle_{\mathcal{D}(\Gamma)}
=i\omega \int_{\Omega} \nabla \times H \cdot \ov{U} 
- H \cdot \nabla \times \ov{U}  \,dx.
\end{align*}
On the set $D$ we have on the other hand that
\begin{align*} 
B_D(E_e,U) = 
\int_{D} \mu^{-1}_e \nabla \times E_e  \cdot \nabla \times  \ov{U} 
- \omega^2 \veps_e E_e \cdot \ov{U} \,dx.
\end{align*}
Thus picking 
\begin{align*} 
\tilde F := 
\left\{
\begin{aligned}
&\quad\mu H/(i\omega) , & \text{ in } \Omega,  & & \\
&-\nabla \times E_e, & \text{ in } D, & &\\
\end{aligned}
\right. \quad
\tilde G := 
\left\{
\begin{aligned}
&\nabla \times H, & \text{ in } \Omega,  & & \\
&i \omega \veps_e E_e, & \text{ in } D, & &\\
\end{aligned}
\right. 
\end{align*}
gives the desired result.
The norm estimate of the claim follows from  Lemma \ref{lem_Ee}.

\end{proof}

\noindent
The two previous lemmas show that $E_e$ and $H_e$ have the necessary properties to
apply Proposition \ref{prop_globProp}, which we will do next in deriving 
the stability estimate for the Cauchy problem and thus proving Theorem \ref{thm_CauchyProb_1}.
We repeat the statement of Theorem \ref{thm_CauchyProb_1} for the convenience of the reader.

\thmTwo*


\begin{proof} Let $E_e,H_e \in \Hspcurl{0}{p}{\Omega \cup D}$ be the vector fields given 
by Lemma \ref{lem_Ee}. By Lemma \ref{lem_Eesol} we see that $E_e$ solves the second order system
$$
L_{\veps,\mu} E_e = i \omega \tilde G - \nabla \times ( \mu^{-1} \tilde F )  \;\in \Hscurl{-1}{\Omega \cup D}.
$$
We will apply the estimate of Proposition \ref{prop_globProp}.
Choose a ball $B(x_0,r_0) \subset \subset D$, and suppose
that $\eta'$ and $\zeta'$ are in accordance with Proposition \ref{prop_globProp},
so that
$$
\eta' \geq  \|E_e\|_{\Hcurl{B(x_0,r_0)}}, \quad 
\quad  \zeta' \geq \|E_e\|_{\Hspcurl{0}{p}{\Omega \cup D}}. 
$$
Then we have by the estimate in Proposition \ref{prop_globProp} on $\Omega \cup D$ that
\begin{align} \label{eq_logestExt}
\|E_e\|_{\Hcurl{\Omega \cup D}} \leq C(\zeta' +  \tilde M_0) 
\omega\Big( \frac{\eta' + \tilde  M_0}{\zeta' + \tilde  M_0} \Big),
\end{align}
where $\omega(t) := C \big( \log ( 1/t) \big)^{-m},\; t<1$, 
and where we pick
$$
\tilde M_0 = \| \tilde F \|_{L^2(\Omega \cup D)^3} +   \| \tilde G \|_{L^2(\Omega \cup D)^3}. 
$$
Now suppose that
$$
\eta \gtrsim\| f \|_{TB_{D}^{-1/p}(\Gamma)} +  \| g \|_{TB_{D}^{-1/p}(\Gamma)}, 
\quad  
\zeta \gtrsim  \| E \|_{L^p(\Omega)^3}  + \| H \|_{L^p(\Omega)^3} 
$$
as in the assumptions of the claim.
We claim that 
\begin{align}\label{eq_etaM0}
C\eta - \tilde M_0 \geq \|E_e\|_{\Hcurl{B(x_0,r_0)}},
\end{align}
for some constant $C>0$.
To see this note firstly that by Lemma \ref{lem_Eesol} 
$$
\tilde M_0 \lesssim \eta.
$$
On the other hand by Lemma \ref{lem_Ee}
$$
\|E_e\|_{\Hcurl{B(x_0,r_0)}} \leq \| E_e \|_{\Hcurl{D}}  \lesssim \eta.
$$
Thus \eqref{eq_etaM0} holds. Next we claim that
\begin{align}  \label{eq_etazetaM0}
C(\eta + \zeta) - \tilde M_0 \geq
\|E_e\|_{\Hspcurl{0}{p}{\Omega \cup D}} 
\end{align}
for some constant $C>0$.
By the estimate in Lemma \ref{lem_Ee} we have
\begin{align*}
\| E_e \|_{ \Hspcurl{0}{p}{\Omega \cup D} }
\lesssim
\| E \|_{ \Hspcurl{0}{p}{\Omega}} + \| f \|_{TB_{D}^{-1/p}(\Gamma)} + \| g \|_{TB_{D}^{-1/p}(\Gamma)} 
\lesssim \zeta + \eta.
\end{align*}
Combining this with the earlier estimate on $\tilde M_0$, shows that
\eqref{eq_etazetaM0} holds.

We can now choose $\eta' = C\eta - \tilde M_0$ as given by \eqref{eq_etaM0} 
and $\zeta' = C(\eta + \zeta) - \tilde M_0$ as given by \eqref{eq_etazetaM0},
and insert these in to estimate \eqref{eq_logestExt} gives that
\begin{align*}
\|E\|_{\Hcurl{\Omega}} \leq C(\zeta + \eta) 
\omega\Big( \frac{\eta}{ \zeta + \eta} \Big).
\end{align*}
which is the first part of the estimate we wanted. The estimate for $H$,
follows from that
$$
\| H \|_{ \Hcurl{\Omega} }
 \leq C \| E \|_{ \Hcurl{\Omega} },
$$
since $(E,H)$ solves the system \eqref{eq_CauchyProbSys2}.

\end{proof}

\section{Runge approximation}%
\label{sec_Runge_approximation}

\noindent
In this second part of the paper we prove Theorem \ref{thm_qrunge}, which gives the 
quantitative Runge approximation for the time-harmonic Maxwell equations
extending the results of \cite{RS}.
We begin by deriving a quantitative form of the unique continuation principle.

\subsection{A quantitative unique continuation result} \label{sec_qucp}
Here we derive a quantitative unique continuation type result for a source problem, which
readily follows from the Cauchy stability estimate of Theorem \ref{thm_CauchyProb_1}
and the $L^p$-estimates for the Maxwell system obtained in \cite{KS} and \cite{KS2}.
This is one of the main tools in the Runge approximation argument in the following section.

Note that the result is formulated for the adjoint system, since this will be used later.
A similar estimate can likewise be obtained for the original Maxwell system. 
Note also that we formulate the estimate of the proposition in a way that allows, 
the source terms $F$ and $\tilde F$ to be in a $L^p$-space that is greater than
in which the solutions appear to be.
This formulation will be useful later when proving Theorem \ref{thm_qrunge} and we need to interpolate.

\begin{prop}  \label{prop_QUCP}
Let $\Omega \subset \R^3$ be a bounded Lipschitz domain, and
let $D \subset \Omega$ be a relatively open subset with Lipschitz boundary, for which
$S:=\Omega \setminus \ov{D}$ is connected.
Suppose that $W,U \in \Hspcurl{0}{p}{\Omega}$, for some $p>2$, 
solve the problem
\begin{align} \label{eq_CauchyProbQUCP}
\left\{
\begin{aligned}
\nabla \times W +  i\omega \mu U &= \tilde F, & & \textnormal{in} \,\Omega,\\
\nabla \times U -  i\omega \veps W &= F, & & \textnormal{in} \,\Omega,\\
\nu \times W &= 0, & & \textnormal{on} \, \p \Omega, \\
\end{aligned}
\right.
\end{align}
where $F,\tilde F \in L^q(\Omega)^3$, $q \geq p$, are such that 
$\supp(F),\supp(\tilde F) \subset \overline{D}$, and
$\Gamma \subset \p \Omega \setminus \ov{D}$ is 
open and non-empty.
Then there are constants $m,C>0$, such that
\begin{align*}
\|W\|_{H(\textnormal{curl}; S)} 
+ \| U \|_{H(\textnormal{curl}; S)} 
\leq 
C\frac{\| F\|_{L^q(D)^3} + \| \tilde F\|_{L^q(D)^3}}
{\Bigg( \log C 
\frac{\|F\|_{L^q(D)^3} + \|\tilde F\|_{L^q(D)^3}}{\|\nu\times U \|_{TB_{D}^{-1/p}(\Gamma)}} 
\Bigg)^{m}}.
\end{align*}
\end{prop}

%
%

\begin{proof}
We begin by rewriting the adjoint system \eqref{eq_CauchyProbQUCP} in a form 
where it is clear that we can apply the Cauchy stability result of the Maxwell system.
Let $E$, $H$, $\tilde \veps$ and $\tilde \mu$ be defined as
$$
E := U,\quad\quad H:= W,\quad\quad \tilde \veps := \mu,\quad\quad \tilde \mu := \veps,\quad\quad \text{ in } \Omega.
$$
The system \eqref{eq_CauchyProbQUCP} implies, then that
\begin{align} \label{eq_nonAdjoint}
\left\{
\begin{aligned}
\nabla \times E -  i\omega \tilde \mu H &= 0, & & \textnormal{in} \,S,\\
\nabla \times H +  i\omega \tilde \veps E  &= 0, & & \textnormal{in} \,S.\\
\nu \times H &= 0, & & \textnormal{on} \, \p \Omega, \\
\end{aligned}
\right.
\end{align}
It is now clear that $(E,H)$ solves a Cauchy problem of the form 
of Theorem \ref{thm_CauchyProb_1} in the set $S$, where 
$g := \nu \times H|_\Gamma = 0$, and $f := \nu \times E |_\Gamma$.
To apply Theorem \ref{thm_CauchyProb_1} we choose
$$
\eta := 
\|f\|_{TB_{D}^{-1/p}(\Gamma)}.
$$
And furthermore 
\begin{align}  \label{eq_zetadef}
\zeta := C (\| \tilde F \|_{ L^q(\Omega)^3}  + \| F \|_{ L^q(\Omega)^3})
- \|f\|_{TB_{D}^{-1/p}(\Gamma)},
\end{align}
where $C>0$ is a large constant, that we will specify shortly. 

Next we check that that $\eta$ and $\zeta$  satisfies the requirements
of Theorem \ref{thm_CauchyProb_1}, when $C$ is large enough. 
The condition on $\eta$ clearly holds, it thus enough to check the condition on the $\zeta$
defined above.

Firstly note that $F,\tilde F \in L^p(\Omega)^3$. 
By  the $L^p$-estimates on the solution of the Maxwell system \eqref{eq_CauchyProbQUCP}
obtained in \cite{KS} Theorem 1 and Remark 1 
(see also \cite{KS2} Theorem 4.5 and remark 4.6 for the
case of more regular boundary)  for $H$, which  solves
$$
\nabla \times (\tilde \veps^{-1} \nabla \times H ) - \omega^2 \tilde \mu H = 
- i\omega F + \nabla \times \big( \tilde \veps^{-1} \tilde F \big), \quad \text{ in } \Omega,
$$
with the boundary condition $(\nu\times H )|_{\p\Omega} = 0$, we get that
\begin{align}  \label{eq_Hest}
\| H \|_{\Hspcurl{0}{p}{\Omega}} 
\leq
C\big(\| \tilde F \|_{ L^p(\Omega)^3}
+ \|  F \|_{ L^p(\Omega)^3}\big).
\end{align}
Notice also that this together with \eqref{eq_nonAdjoint} implies that
\begin{align}  \label{eq_Eest}
\| E \|_{\Hspcurl{0}{p}{\Omega}} 
\leq
C\big(\| \tilde F \|_{ L^p(\Omega)^3}
+ \|  F \|_{ L^p(\Omega)^3}\big).
\end{align}
The above inequality and the continuity of the tangential trace, Lemma \ref{lem_trace}, gives that
$$
\|f \|_{TB_{D}^{-1/p}(\Gamma)} 
=
\|\nu\times E\|_{TB_{D}^{-1/p}(\Gamma)} 
\leq  
C\big(\| \tilde F \|_{ L^q(\Omega)^3} + \| F \|_{ L^q(\Omega)^3} \big),
$$
since $q \geq p$. This together with \eqref{eq_Hest} and \eqref{eq_Eest}, 
gives then that
\begin{align*} 
\| E \|_{ \Hspcurl{0}{p}{\Omega} }
+ \| H \|_{ \Hspcurl{0}{p}{\Omega} }
+ \|f \|_{TB_{D}^{-1/p}(\Gamma)} 
\leq C 
(\| \tilde F \|_{ L^q(\Omega)^3} + \| F \|_{ L^q(\Omega)^3}).
\end{align*}
So that picking $C$ in \eqref{eq_zetadef} as the constant $C$ for which the above inequality holds,
we get that
\begin{align*} 
\zeta 
\geq
\| E \|_{ \Hspcurl{0}{p}{\Omega} }
+ \| H \|_{ \Hspcurl{0}{p}{\Omega} }.
\end{align*}
Our choices of $\eta$ and $\zeta$ satisfy therefore the requirements of Theorem
\ref{thm_CauchyProb_1}, and we thus have the estimate
\begin{align*} 
\|E\|_{\Hcurl{S}}  + \|H\|_{\Hcurl{S}}
\leq C 
\frac{\| \tilde F \|_{ L^q(\Omega)^3} + \| F \|_{ L^q(\Omega)^3}}
{\Bigg( \log C \frac{\| \tilde F \|_{ L^q(\Omega)^3} + \| F \|_{ L^q(\Omega)^3}}
{\|\nu\times E \|_{TB_{D}^{-1/p}(\Gamma)}} \Bigg)^{m}},
\end{align*}
which directly implies the estimate of the claim.

\end{proof}

\noindent
To see why Proposition \ref{prop_QUCP} gives a 
quantitative version of the unique continuation principle, notice
that one gets that $(W,U)$ vanishes in $S$, when  
the norm of $\nu \times U |_{\Gamma}$ tends to zero.

\subsection{Quantitative Runge approximation}\label{sec_qrunge}
We are now ready to prove Theorem \ref{thm_qrunge}. The theorem is a direct consequence of Proposition
\ref{prop_QUCP} below.  

\medskip
\noindent
Define  the subspace $\mathcal{V} \subset TH_D^{-1/2}(\p \Omega) $, as 
$$
\mathcal{V} := \ov{\mathcal{V}_0},
$$
where the closure is in the $TH_D^{-1/2}(\p \Omega)$-norm, and
$$
\mathcal{V}_0 := \{ f \in TH_D^{-1/2}(\p \Omega) \;:\; \supp(f)\subset \Gamma \}.
$$
Furthermore we define the space
$$
\mathcal{X} := \ov{\mathcal{S}_A},
$$
where the closure is in the $L^2(A)^6$-norm, and
where $\mathcal{S}_A$ is the set in \eqref{eq_SA}, i.e.
$$
\mathcal{S}_A =\big \{ (E,H) \in \Hcurl{A}^2 \;:\; \mathcal{M}(E,H)=0, 
\text{ in } A \big \}. 
$$
We define the operator  $\mathcal{A}: \mathcal{V} \to \mathcal{X}$, 
as
$$
\mathcal{A}: f \mapsto  \big(E_f|_A,\, H_f|_A \big),
$$
where $(E_f,H_f)$ solves the system
\begin{align}\label{eq_Asys} 
\left\{
\begin{aligned}
\nabla \times E_f -  i\omega \mu H_f &= 0, & & \textnormal{in} \,\Omega,\\
\nabla \times H_f +  i\omega \veps E_f &= 0, & & \textnormal{in} \,\Omega,\\
\nu \times E &= f, & & \textnormal{on} \, \p\Omega. \\
\end{aligned}
\right.
\end{align}

\noindent
The next lemma gives a characterization of the  dual  $\mathcal{V}^*$.
This characterization essentially follows from the fact that
the dual of $TH_D^{-1/2}(\p\Omega)$ is isomorphic to $TH_C^{-1/2}(\p\Omega)$.
For more on this see \cite{KH} and \cite{C}.
We will need the following space
$$
\mathcal{W} := \ov{\mathcal{W}_0},
$$
where the closure is taken in the $TH_C^{-1/2}(\p \Omega)$-norm, and
$$
\mathcal{W}_0 := \{ g \in TH_C^{-1/2}(\p \Omega) \;:\; \supp(g)\subset \Gamma \}.
$$

\begin{lem} \label{lem_Vdual2}
We have the isomorphisms 
$$
\mathcal{I} : \mathcal{W} \eqsim \mathcal{V}^*,  \quad
\mathcal{J} : \mathcal{V} \eqsim \mathcal{W}^*,  \quad
\mathcal{I}: g \mapsto i_g(\cdot),\quad
\mathcal{J}: f \mapsto j_f(\cdot),
$$
where $i_g$ and $j_f$ are defined by 
\begin{align*} 
i_g(h) &:= 
\int_{\Omega} \nabla \times \eta_Tg  \cdot \ov{\eta_t h}  \,dx 
- \int_{\Omega} \eta_T g \cdot \nabla \times \ov{\eta_t h}  \,dx,
\\
j_f(h) &:= 
- \int_{\Omega} \nabla \times \eta_Th  \cdot \ov{\eta_t f}  \,dx 
+ \int_{\Omega} \eta_T h \cdot \nabla \times \ov{\eta_t f}  \,dx.
\end{align*}
\end{lem}

\begin{proof}
The open mapping Theorem implies that it is enough to check that $\mathcal{I}$ and $\mathcal{J}$
are continuous, injective and surjective. 
To prove the continuity of $\mathcal{I}$, we let $g \in C_0^\infty(\Gamma)^3 \cap \{ \nu \times g|_{\p \Omega} = 0 \}$.
These functions are dense in $\mathcal{W}$.
Let $X:= \{  \varphi \in C^\infty_0(\Gamma)^3\cap \mathcal{V}\,:\, \| \varphi \|_{\mathcal{V}} = 1 \}$. 
Continuity follows, from the estimate 
\begin{align*} 
\| \mathcal{I} g \|_{\mathcal{V}^*}
&= \sup_{ \varphi \in X} \big|  i_g(\varphi) \big| \\
&\leq 
C \sup_{\varphi \in X} 
\big(\| \eta_Tg  \|_{ \Hcurl{\Omega} } \| \eta_t\varphi  \|_{ \Hcurl{\Omega} }\big)
\\
&\leq
C \| g \|_{ TH_C^{-1/2}(\p \Omega)}.
\end{align*}
The continuity of $\mathcal{J}$ is proved similarly.

\medskip
Next we show that $\mathcal{I}$ is surjective. The proof that $\mathcal{J}$ is surjective is similar and
is omitted.
To this end we let $\ell \in \mathcal{V}^*$. 
Let
$$
Z :=\{ U \in \Hcurl{\Omega} \,:\, \supp (\nu \times U|_{\p\Omega}) \subset \subset \Gamma \}. 
$$
The closure $\ov{Z}$ in the $\Hcurl{\Omega}$-norm is a closed subspace of $\Hcurl{\Omega}$.
The formula $\ell \circ (\nu \times \,\cdot\,)|_{\p\Omega}$, defines then a
 linear functional in $\ov{Z}^*$.
By the Riesz representation Theorem there is an $U_\ell \in \ov{Z}$, such that
$$
\ell \circ (\nu \times \,\cdot\,)|_{\p\Omega} = \langle  U_\ell,\,\cdot\, \rangle_{\Hcurl{\Omega}}.
$$
Taking  $\varphi \in H_0(\operatorname{curl}\,;\Omega)$,  and using the integration by parts formula
shows that
\begin{align}  \label{eq_Ul}
\nabla \times (	\nabla \times U_\ell) = - U_\ell.
\end{align}
Thus for any $f \in \mathcal{V}$, 
we have using \eqref{eq_Ul} that
\begin{align*} 
\ell (f) = 
\ell \circ (\nu \times \eta_t f \,)|_{\p\Omega} 
= 
\int_{\Omega} \nabla \times U_\ell \cdot \nabla \times \ov{\eta_t f}   
- \nabla \times (\nabla \times U_\ell) \cdot \ov{\eta_t f}   \,dx.
\end{align*}
By \eqref{eq_Ul} we have that 
$g :=  \nu \times (\nu \times (\nabla \times U_\ell))|_{\p \Omega} \in TH^{-1/2}_C(\p\Omega)$,
and by the integration by parts formula
\begin{align*}  
\int_{\Omega} \nabla \times U_\ell \cdot \nabla \times \ov{\eta_t f}
- \nabla \times (\nabla \times U_\ell) \cdot \ov{\eta_t f}   \,dx 
&= 
\int_{\Omega} \nabla \times \eta_T g \cdot \ov{\eta_t f} 
- 
\eta_T g \cdot \nabla \times \ov{\eta_t f}   \,dx \\ 
&=
i_g(f).
\end{align*}
Thus $\ell(f) = i_g(f) = \mathcal{I}g(f)$ and $\mathcal{I}$ is surjective.

\medskip
Finally we prove  the injectivity of $\mathcal{I}$ and $\mathcal{J}$. 
We give the proof for $\mathcal{I}$, the proof for $\mathcal{J}$ is similar. 
We will use the surjectivity of $\mathcal{J}$ and the fact that
$$
i_g(f) = j_f(g),
$$
which follows from the definitions. Suppose that $\mathcal{I}$ is not injective, so 
that there is an $g\neq 0$, with $i_g \equiv 0$. Define the functional
$$
\ell ( h) := \mathcal{L} \circ \operatorname{proj}_g(h),
\quad
\mathcal{L}(c \hat g) := c,
\quad \hat g := g /  \| g \|_{ TH_C^{-1/2}(\p\Omega) },
$$
where $h \in \mathcal{W}$.
Clearly $\mathcal{L} : \operatorname{span}\{g\} \to \R$ is an isomorphism. 
By the surjectivity
of $\mathcal{J}$, there is a $j_f$, such that $j_f = \ell$ and thus
$$
\| g \|_{ TH_C^{-1/2}(\p\Omega) } =  \ell(g) = j_f(g) = i_g(f) = 0,
$$
which is a contradiction, since $g \neq0$.

\end{proof}

\noindent
We need the Hilbert space adjoint $\mathcal{A}^* :\mathcal{X} \to \mathcal{V}$,
to obtain a singular value decomposition of $\mathcal{A}$.
This is given by the next Lemma.

\begin{lem} \label{lem_Astar}
The Hilbert space adjoint 
$\mathcal{A}^* : \mathcal{X} \to \mathcal{V}$, is given by
$$
\mathcal{A}^* : F =: (F_1, F_2) 
\mapsto \mathcal{R}\circ \mathcal{I}\,(\nu \times(\nu\times U_F) ) 
$$
where 
$\mathcal{R} : \mathcal{V}^* \to \mathcal{V}$ is the Riesz isomorphism, and 
$\mathcal{I}$ is the isomorphism of Lemma \ref{lem_Vdual2}
and where $(U_F,W_F)$ solves 
\begin{align}\label{eq_Astarsys}
\left\{
\begin{aligned}
\nabla \times W_F +  i\omega \mu U_F &= F_2, & & \textnormal{in} \,\Omega,\\
\nabla \times U_F -  i\omega \veps W_F &= F_1, & & \textnormal{in} \,\Omega,\\
\nu \times W_F&= 0, & & \textnormal{on} \, \p\Omega, \\
\end{aligned}
\right.
\end{align}
where $F_1$ and $F_2$ are extended by zero in $\Omega \setminus A$.
\end{lem}

\begin{proof}
Let $E_f$ and $H_f$ be as in \eqref{eq_Asys}. Then using the integration by parts formula \eqref{eq_intByParts}, 
the mapping $\mathcal{I}$ of Lemma \ref{lem_Vdual2}, and that fact that $\mathcal{R}$
is a Riesz isomorphism, we see that
\begin{align*}
\big( \mathcal{A} f , F \big)_{L^2(A)^6}
&=
\big( E_f , F_1 \big)_{L^2(A)^3} + \big( H_f , F_2 \big)_{L^2(A)^3} \\
&=
\big \langle f , \; \nu \times (\nu \times  U_F) \big \rangle_{\mathcal{D}(\p\Omega)} \\
&=
\mathcal{I}\,\big(\nu \times (\nu \times  U_F) \big) (f) \\
&=
\big( \mathcal{R} \circ \mathcal{I}\,(\nu \times (\nu \times  U_F) ), \,f \,)_{TH_D^{-1/2}(\p \Omega)}.
\end{align*}
The Hilbert space adjoint of $\mathcal{A}$ is thus given by
$$
\mathcal{A}^* : F \mapsto \mathcal{R} \circ \mathcal{I}\,(\nu \times (\nu \times  U_F) ).
$$

\end{proof}

\medskip
\noindent
The solutions $(E,H)$ to the Maxwell system \eqref{eq_MEsys} are guaranteed to be only in $\Hcurl{\Omega}$, if
the boundary data is in $TH_D^{-1/2}(\p\Omega)$. Interior elliptic regularity assures however that
$(E,H)|_A \in H^1(A)^6$, for $A \subset \subset \Omega$. We will shortly use this fact to obtain
the singular value decomposition of the operator $\mathcal{A}$.
The next lemma gives an elliptic regularity result of this sort, adapted to our purposes.
For more on elliptic regularity in the Maxwell case see e.g. \cite{AC} and \cite{W}.

\begin{lem} \label{lem_intReg}
Suppose $\Gamma \subset \p\Omega$ is an open set, and
assume that $\veps,\mu \in W^{1,\infty}(\Omega; \R^{3 \times 3})$ are such that
\eqref{eq_ellip} holds.
Let $(E,H) \in \Hcurl{\Omega}^2$ solve 
\begin{align} \label{eq_EllipReg}
\left\{
\begin{aligned}
\nabla \times E -  i\omega \mu H &= 0, & & \textnormal{in} \,\Omega,\\
\nabla \times H +  i\omega \veps E &= 0, & & \textnormal{in} \,\Omega,\\
(\nu \times E)|_\Gamma &= f, & & \textnormal{on} \, \p \Omega, \\
\end{aligned}
\right.
\end{align}
where $f \in TH_D^{-1/2}(\p\Omega)$, $\supp(f)\subset \subset \Gamma$.
Let $S \subset \subset \Omega$.
Then we have the estimate
$$
\| (E,H) \|_{H^1(S)^6} 
\leq C
\| f \|_{TH_D^{-1/2}(\p\Omega)}.
$$
\end{lem}

%
%

\begin{proof} 
It is enough to prove the claim  locally for a ball $B \subset \subset \Omega$, 
since $S$ is compactly contained in $\Omega$.
We begin by picking two additional balls $B'$ and $B''$, s.t. 
$$
B \subset \subset B' \subset \subset B'' \subset \subset \Omega.
$$
We will use a Helmholtz decomposition to prove the claim. 
The ball $B''$ is simply connected, and thus by
Theorem 2, section 2 in \cite{S}, we can decompose the field $E$ on $B''$, as 
$$
E = \curl \Phi + \nabla \varphi,\quad \varphi \in H^1(B''), \; \Phi \in H^1(B'')^3,\; \nabla \cdot \Phi =0,
$$
where additionally $(\nu \times \Phi)|_{\p B''} = 0$. And where we have the estimate
\begin{align}  \label{eq_HelmholtzEst}
    \| \nabla \varphi \|_{ H^1(B'') } + \| \Phi \|_{ H^1(B'')^3 } \leq C \| E \|_{ L^2(B'')^3}.
\end{align}
The curl in the first  equation of \eqref{eq_EllipReg} gives, 
using the identity $\Delta \Phi = \nabla \times (\nabla \times \Phi) + \nabla (\nabla \cdot \Phi) $,
the equation
$$
\Delta \Phi = i\omega \mu H, \quad \text{ in } B''.
$$
Since $ i \omega \mu H \in L^2(B'')^3$ and $B' \subset \subset B''$, we get using the  
interior elliptic regularity result for scalar equations in \cite{GT} (see Theorem 8 p.183),
componentwise on $\Phi$, that $\Phi \in H^2(B')^3$, and that
\begin{align}  \label{eq_reg1}
\| \Phi \|_{H^2(B')^3} 
&\leq C \big( \| \Phi \|_{ H^1(B'')^3 }  + \| i\omega \mu H \|_{ L^2(\Omega)^3 } \big) \nonumber \\
&\leq C \big( \| E \|_{ L^2(\Omega)^3 }  + \| \nabla \times E \|_{ L^2(\Omega)^3 }\big) \\
&\leq C \| f \|_{TH_D^{-1/2}(\p\Omega)},\nonumber 
\end{align}
where in the second inequality we used \eqref{eq_HelmholtzEst}
and the last inequality follows from the well-posedness of the problem \eqref{eq_EllipReg}.

By taking the divergence of the second equation in \eqref{eq_EllipReg}, we get that
$$
\nabla \cdot (\eps \nabla \varphi) = -\nabla \cdot (\eps \nabla \times \Phi), \quad \text{ in } B'.
$$
By the above we know that $\Phi \in H^2(B')^3$, 
and thus $\nabla \cdot (\eps \nabla \times \Phi)\in L^2(B')$. Now since $B \subset \subset B'$,
we have again by 
the elliptic regularity result, Theorem 8, p.183 in \cite{GT}, that $\varphi \in H^2(B)$,
and
\begin{align}  \label{eq_reg2}
\| \varphi \|_{H^2(B)} 
&\leq C \big( \| \varphi \|_{ H^1(B'')  } + \| \nabla \eps \cdot \nabla \times \Phi  \|_{ L^2(B'')^3 } \big) \nonumber\\
&\leq C \| E  \|_{ L^2(\Omega)^3 } \\
&\leq C \| f \|_{TH_D^{-1/2}(\p\Omega)}, \nonumber
\end{align}
Thus by \eqref{eq_reg1} and \eqref{eq_reg2},
and since $E = \curl \Phi + \nabla \varphi$, we get that
$$
\| E \|_{H^1(B)^3} \leq C \| f \|_{TH_D^{-1/2}(\p\Omega)}.
$$
The claim follows from doing the identical argument for $H$ as was done for $E$.

\end{proof}

\noindent
Next we  obtain the singular value decomposition of the operator $\mathcal{A}$.

\begin{lem} \label{lem_SVD}
The operator $\mathcal{A}: \mathcal{V} \to \mathcal{X}$ has 
a singular value decomposition. We can more specifically decompose $\mathcal{A}$ as
$$
\mathcal{A} (\,\cdot\,) = \sum_j \sigma_j (\varphi_j, \;\cdot\; )_{TH_D^{-1/2}} \Psi_j.
$$
where  $\{ \varphi_j$\;:\; $j\in \N \}$ and 
$\{\Psi_j$ \;:\; $j \in \N\}$ are orthonormal bases of $\mathcal{V}$ and $\mathcal{X}$,
and the corresponding $\sigma_j >0$, are s.t.  $\sigma_j \to 0$.
\end{lem}

\begin{proof} We begin by showing that $\mathcal{A}$ is a compact operator. 
Let $(g_j) \subset \mathcal{V}$ be a bounded sequence. By Lemma \ref{lem_intReg} we see that  $\mathcal{A}$ is continuous
and that the sequence $(\mathcal{A} g_j)$ is also bounded in $H^1(A)^6$.
The Rellich-Kondrachov Theorem gives a converging subsequence in $L^2(A)^6$.
Let the limit in the $L^2(A)^6$-norm be $\tilde F$, then $\tilde F \in \ov{\mathcal{X}}=\mathcal{X}$, 
and $\mathcal{A}$ is hence compact.

A priori estimates and the continuity of the tangential trace shows that the 
adjoint operator $\mathcal{A}^*$ is bounded. It follows that the map 
$\mathcal{A}^*\mathcal{A}:\mathcal{V}\to \mathcal{ V}$ 
is compact. It is also 
self adjoint. The Hilbert-Schmidt Theorem implies that there is a complete
orthonormal basis $\{ \varphi_j\} \subset \mathcal{V}$ and corresponding $\lambda_j\to 0$,
s.t. 
$$
\mathcal{A^*}\mathcal{A} (\,\cdot\,)= \sum_j \lambda_j (\varphi_j, \;\cdot\; )_{TH_D^{-1/2}} \varphi_j.
$$
This operator is positive so $\lambda_j > 0$. Now set $\sigma_j := \lambda_j^{-1/2} > 0$,
and $\Psi_j := \sigma_j^{-1}\mathcal{A}\varphi_j$. 
Now let $f \in \mathcal{V}$, and consider the formula of the statement
of the lemma, then
\begin{align*}
\sum_j \sigma_j (\varphi_j, f )_{TH_D^{-1/2}} \Psi_j
=
\sum_j (\varphi_j, f )_{TH_D^{-1/2}} \mathcal{A} \varphi_j
= \mathcal{A} f,
\end{align*}
which shows that formula holds.
It straight forward to check that $\{\Psi_j\}$ is an orthogonal set in $L^2(A)^6$.
The span of  $\{\sigma_j^{-1}\mathcal{A}\varphi_j\}$ is dense, which follows from by
qualitative Runge approximation (see Theorem 4.1 in \cite{HLL}).
The set $\{ \Psi_j \}$ is hence also complete in the set $\mathcal{X}$. 

\end{proof}

\noindent
We are now ready to prove Theorem \ref{thm_qrunge}. 
We repeat the statement of the Theorem for the convenience of the reader.

\thmOne*

\begin{proof}
We use the singular value decomposition 
$$
\mathcal{A}(\;\cdot\;) = \sum_k \sigma_k (\;\cdot\;,\varphi_k)_{TH_D^{-1/2}} \Psi_k,
$$ 
given by Lemma \ref{lem_SVD}.
The field $W := (E,H)$ can be written as the series $W =: \sum_k c_k \Psi_k$.
We set 
$$
R_\alpha W := \sum_{\sigma_k \geq \alpha} \frac{c_k}{\sigma_k} \varphi_k \in \mathcal{V}.
$$
Next pick $(E_\alpha, H_\alpha)$, so that $(E_\alpha|_A,H_\alpha|_A) = \mathcal{A}R_\alpha W$, 
i.e. let $(E_\alpha,H_\alpha)$ solves the system
\begin{align} \label{eq_EHsys} 
\left\{
\begin{aligned}
\nabla \times E_\alpha -  i\omega \mu H_\alpha &= 0, & & \textnormal{in} \,\Omega,\\
\nabla \times H_\alpha +  i\omega \veps E_\alpha &= 0, & & \textnormal{in} \,\Omega,\\
\nu \times E_\alpha &= R_\alpha W, & & \textnormal{on} \, \p\Omega, \\
\end{aligned}
\right.
\end{align}
The fields $(E_\alpha,H_\alpha)$ will give us the desired approximation.
Before showing this we additionally  define $r_\alpha$ as 
\begin{align} \label{eq_ra} 
r_\alpha = (r_{\alpha,1},r_{\alpha,2}) := W- (E_\alpha |_A, H_\alpha|_A) = \sum_{\sigma_k < \alpha} c_k \Psi_k. 
\end{align}
We will denote the extension of $r_\alpha$ by zero to $\Omega \setminus A$, also by $r_\alpha$. Note that
$r_\alpha \in L^q(\Omega)^6$, with $q>2$, since we assume that $W \in L^p(A)^6$,
and by interior regularity result in Lemma \ref{lem_intReg}
and Sobolev embedding we have that $(E_\alpha |_A, H_\alpha|_A) \in H^1(A)^6 \subset L^6(A)^6$.
Thus we have that $r_\alpha \in L^q(\Omega)^6$, $q = \min(p,6)$.

Finally define $(V_\alpha,U_\alpha)$ to be the solution  of the adjoint source problem
\begin{align*} 
\left\{
\begin{aligned}
\nabla \times V_\alpha +  i\omega \mu U_\alpha &= r_{\alpha,2}, & & \textnormal{in} \,\Omega,\\
\nabla \times U_\alpha -  i\omega \veps V_\alpha &= r_{\alpha,1}, & & \textnormal{in} \,\Omega,\\
\nu \times V_\alpha &= 0, & & \textnormal{on} \, \p\Omega, \\
\end{aligned}
\right.
\end{align*}
Note that 
$\mathcal{A}^* r_\alpha$ is determined by $\nu\times (\nu \times  U_\alpha)$. Note also that
since $r_\alpha \in L^q(A)^6$, $q \in (2,\min(6,p))$, we have by the $L^p$-estimates 
of Theorem 1 in \cite{KS2} that
$(V_\alpha,U_\alpha) \in \Hspcurl{0}{q}{\Omega}$, for $q \in (2,\min(6,p))$.
Now consider
\begin{align*}
\big\| E - E_\alpha|_A \big\|^2_{L^2(A)^3} &= \|r_{\alpha,1} \|^2_{L^2(A)^3} 
= ( E,r_{\alpha,1})_{L^2(A)^3} \\
&= ( E, \nabla \times U_\alpha -  i\omega \veps V_\alpha)_{L^2(A)^3}. 
\end{align*}
Likewise consider
\begin{align*}
\big\| H - H_\alpha|_A  \big\|^2_{L^2(A)^3} 
= ( H, \nabla \times V_\alpha +  i\omega \mu U_\alpha)_{L^2(A)^3}. 
\end{align*}
Adding these and integrating by parts gives
\begin{align*}
\big\|&  E_\alpha|_A -  E  \big\|^2_{L^2(A)^3} 
+ 
\big\|  H_\alpha|_A -  H  \big\|^2_{L^2(A)^3} 
= 
\big\langle \nu\times  E, U_\alpha\big \rangle_{\mathcal{D}(\p A)}
+ \big\langle \nu\times  H, V_\alpha\big\rangle_{\mathcal{D}(\p A)},
\end{align*}
where the terms on the last line are the duality pairings between $TH_D^{-1/2}(\p A)$
and its dual.
We estimate the second boundary term using the integration by parts formula 
\eqref{eq_intByParts}.
Let $H_e$ be a continuous extension of $H$ to $\Hcurl{\Omega \setminus A}$,
that is supported away from $\p\Omega$, then 
\begin{align*}
\big\langle \nu\times H, V_\alpha\big \rangle_{\mathcal{D}(\p A)}
&= 
\int_{\Omega \setminus A} 
\nabla \times H_e \cdot \ov{V}_\alpha  \,dx -
H_e \cdot \nabla \times \ov{V_\alpha}  \,dx  \\
&\leq 
\| H_e \|_{\Hcurl{\Omega \setminus A}  } \| V_\alpha \|_{ \Hcurl{\Omega \setminus A}} \\
&\leq 
C \| H \|_{\Hcurl{A}  } \| V_\alpha \|_{ \Hcurl{\Omega \setminus A}}. 
\end{align*}
We can estimate the other boundary term similarly as
\begin{align*}
\big\langle \nu\times E, U_\alpha\big \rangle_{\mathcal{D}(\p A)}
&\leq 
C \| E \|_{\Hcurl{A}  } \| U_\alpha \|_{ \Hcurl{\Omega \setminus A}}.
\end{align*}
We can now apply  Proposition 
\ref{prop_QUCP} to estimate the norms of $U_\alpha$ and $V_\alpha$ over the set $\Omega \setminus A$
in the two previous estimates. By choosing $q_0 \in \big(q,\min(p,6)\big)$ 
and using Proposition \ref{prop_QUCP}, we obtain that
\begin{align*}
\big\|  (E_\alpha,H_\alpha) &- W \big\|^2_{L^2(A)^6}
\leq 
C  \big\| ( E, H) \big\|_{\Hcurl{A}^2}\|r_\alpha\|_{L^{q_0}(A)^6}
\Bigg( \log  
\frac{C\|r_\alpha\|_{L^{q_0}(A)^6}}{\|\nu\times U_\alpha\|_{TB^{-1/q}_{D}(\Gamma)}} \Bigg)^{-m}.
\end{align*}
Next we estimate the boundary term occurring in this expression.
Firstly note that by \eqref{eq_ra} we have that
$$
\mathcal{A}^* r_\alpha = \sum_{\sigma_k < \alpha} c_k \sigma_k \varphi_k.
$$
By Lemma \ref{lem_Astar}, we have that
\begin{equation}  \label{eq_interp1}
\begin{aligned}
\|\nu \times U_\alpha\|^2_{TH_D^{-1/2}(\Gamma)}
&= 
\|\nu  \times (\nu \times U_\alpha)\|^2_{TH_D^{-1/2}(\operatorname{curl},\,\Gamma)}  \\
&=
\|\mathcal{R} \circ \mathcal{I}\, (\nu\times (\nu\times  U_\alpha))\|^2_{TH_D^{-1/2}(\Gamma)} \\
&= 
\| \mathcal{A}^* r_\alpha \|^2_{TH_D^{-1/2}(\Gamma)} \\
&= \sum_{\sigma_k < \alpha} c^2_k \sigma_k^2 \\
&\leq
\alpha^2 \|r_\alpha \|^2_{L^2(A)^6}.
\end{aligned}
\end{equation}
On the other hand we have by the apriori estimates 
on the solution pair $(V_\alpha, U_\alpha)$ (see Theorem 1 and Remark 3.1 in \cite{KS}) 
and the continuity of the tangential trace operator, Lemma \ref{lem_trace}, that
\begin{align} \label{eq_interp2}
\|\nu \times U_\alpha\|_{TB_{D}^{-1/q_0}(\Gamma)}
\leq C \| r_\alpha \|_{ L^{q_0}(A)^6}.
\end{align}
We can interpolate between  \eqref{eq_interp1} and \eqref{eq_interp2},
using the interpolation property of Besov spaces \eqref{eq_Binterp}, to obtain
\begin{align*} 
\|\nu \times U_\alpha\|_{TB_{D}^{-1/q}(\Gamma)}
&\leq C 
\alpha^{1-\theta}\|r_a\|^{1-\theta}_{L^2(A)^6}\| r_\alpha \|^\theta_{ L^{q_0}(A)^6} \\
&\leq C 
\alpha^{1-\theta}\| r_\alpha \|_{ L^{q_0}(A)^6}.
\end{align*}
where $\theta \in (0,1)$ is s.t. $1/q = (1-\theta)/2 + \theta/q_0$, and where $q \in (2,q_0)$.
Hence we obtain that
\begin{align} \label{eq_loglim}
\big\|  (E_\alpha,H_\alpha) - W \big\|^2_{L^2(A)^6}
&\leq
C  \big\| ( E, H) \big\|^2_{\Hspcurl{0}{q_0}{A}^2} 
\Big( \log  \tfrac{C}{\alpha^{1-\theta}} \Big)^{-m} 
\to 0 ,
\end{align}
as $\alpha \to 0$. Notice that since $A$ is bounded, we can replace
the requirement that $q_0 \in (q, \min(p,6))$, $q>2$, by $q_0 \in (q, p\,]$, $q >2$ in the above estimate.
We now choose $\alpha$, so that
\begin{align} \label{eq_jdef}
\tfrac{1}{j} = \Big( \log  \tfrac{C}{\alpha^{1-\theta}} \Big)^{-m/2},
\end{align}
which gives the $(E_j,H_j)$ whose $L^2$-norms satisfy the first estimate of the claim.
We still need to obtain an estimate for the $L^2$-norm of the curl of $(E-E_j,H-H_j)$.
To this end note that $\nabla \times (E_\alpha-  E) - i \omega \mu (H_\alpha - H) = 0$
in $A$, so that
$$
\| \nabla \times (E_j -  E)  \|_{L^2(A)^3} \leq C \| H_j -  H  \|_{L^2(A)^3},
$$
which in conjunction with \eqref{eq_loglim} gives the desired estimate for the entire 
$\Hcurl{\Omega}$-norm of $E_j$.
The term  $\| \nabla \times (H_j|_A -  H)  \|_{L^2(A)^3}$ can be estimated likewise.
The first estimate of the claim follows thus from \eqref{eq_loglim}.

\medskip
\noindent
In order to prove the second part of the claim we argue as follows. Consider
\begin{align*}
\|\nu \times E_\alpha \|^2_{TH_D^{-1/2}(\p \Omega)}
&=
\| R_\alpha W\|^2_{TH_D^{-1/2}(\p\Omega)} \\
&=
\Big \| \sum_{\sigma_k \geq \alpha} \frac{c_k}{\sigma_k} \varphi_k  \Big \|^2_{TH_D^{-1/2}(\p\Omega)} \\
&\leq
\frac{1}{\alpha^2} \| W\|^2_{L^2(A)^6},
\end{align*}
where the last steps follows from the continuity of $\mathcal{A}$.
From the trace theorem and \eqref{eq_jdef}, it follows that 
\begin{align*}
\|\nu \times E_j\|_{TH_D^{-1/2}(\p \Omega)}
\leq C  e^{C j^{2/m}}\| W\|_{L^2(A)^6}.
\end{align*}
By apriori estimates for the system $(E_j,H_j)$ (see Theorem 2.1 in \cite{HLL}) we get that
\begin{align*}
\| (E_j,H_j) \|_{\Hcurl{\Omega}^2}
\leq C  e^{C j^{2/m}}\| W \|_{L^2(A)^6},
\end{align*}
which proves the second estimate of the claim.

\end{proof}

\section{Appendix: Function spaces}\label{sec_fspaces}

\noindent
In this section we gather some of the definitions, notations and
facts relating to function spaces used through out the paper
(for further details see \cite{M2,JK,BCS,C,Mc1}).

\medskip
\noindent
We begin by specifying some standard functions spaces in the scalar case for a Lipschitz domain
$\Omega$.
The $L^2(\Omega)$ based Sobolev spaces are denoted here by $H^s(\Omega)$, $s\in \R$. 
More generally $W^{s,p}(\Omega)$ denotes the fractional Sobolev space
of $L^p(\Omega)$, $1\leq p \leq \infty$ functions with smoothness $s \in \R$.
We define these following \cite{JK} section 2, as the spaces $L^p_s(\Omega)$, which 
we here denote by  $W^{s,p}(\Omega)$, instead of $L^p_s(\Omega)$.
Note also that $H^s(\Omega) = W^{s,2}(\Omega)$.
The space $W^{1,\infty}(\Omega)$ can be identified with Lipschitz continuous functions.
Moreover we will use the notation $Z^*$ to denote the dual of a Banach space $Z$.

When discussing  the boundary traces  we will need  the Besov type spaces $B^s_{p,p}(\p\Omega)$
on a Lipschitz boundary.
These spaces are discussed in more detail in \cite{JK}. We define
$B^s_{p,p}(\p\Omega)$ following \cite{M2} by real interpolation
$$
B^s_{p,p} (\p\Omega) :=  
\big(L^p(\p\Omega), W^{1,p}(\p\Omega)\big)_{s,p}, \quad 0 < s < 1,\quad
1 < p < \infty,
$$
for more details see section 2.2 in \cite{M2}, p.173 in \cite{JK} and \cite{BL}. 
Note also that $H^{s}(\p\Omega) = B^{s}_{2,2}(\p\Omega)$ (this is a consequence of
Theorem 6.4.4. in \cite{BL}).
The spaces $B_{p,p}^{-s}$ can be defined as the dual spaces
$$
B^{-s}_{p,p} (\partial \Omega) :=  (B^{s}_{p'p'} (\partial \Omega) )^*, 
\quad 0 < s < 1,\quad
1 < p < \infty,
$$
where $1/p+1/p'=1$. We will moreover need to interpolate between the above type Besov spaces.
By real interpolation we have that 
\begin{align}  \label{eq_Binterp}
B^{-s_\theta}_{p_\theta,p_\theta} (\partial \Omega) 
= \big(B^{-s_0}_{p_0,p_0} (\partial \Omega),\,B^{-s_1}_{p_1,p_1} (\partial \Omega) \big)_{\theta ,p_\theta},
\end{align}
where $s_\theta=\theta s_0 + (1-\theta)s_1$ and $1/p_\theta = \theta/p_0 + (1-\theta)/p_1$
this is a consequence of Theorem 6.4.5 in \cite{BL}.

\medskip
\noindent
Next we will define a number of function spaces related to vector fields.
If $Z$ is a Banach space, then we use the notation 
$$
Z^n = Z \times \dots \times Z,
$$
where the Cartesian product on the right hand side contains $n$ copies of $Z$.
We also use the notation $W^{1,\infty}(\Omega;\R^n \times \R^n)$ denotes the space of $m\times m$ matrix fields with 
Lipschitz continuous coefficients.

One of the important spaces of vector fields, when dealing with Maxwell's equations is
\begin{align*} 
\Hspcurl{s}{p}{\Omega}
:=
\{ U \in  W^{s,p} ( \Omega)^3 \,:\, \nabla \times  U \in  W^{s,p} ( \Omega)^3 \},
\end{align*}
where, $s \geq 0$, $1 < p < \infty$ and $W^{s,p}(\Omega)$ are the 
fractional Sobolev spaces, and $\Omega \subset \R^3$ is a bounded Lipschitz domain.
The norm is given by
$$
\| U \|_{ \Hspcurl{s}{p}{\Omega}} := 
\big(\| U \|^2_{ W^{s,p}(\Omega)^3 } 
+ \| \nabla \times  U \|^2_{ W^{s,p}(\Omega)^3 }\big)^{1/2}.
$$
We will use the abbreviations
\begin{align*}
\Hcurl{\Omega} &:= \Hspcurl{0}{2}{\Omega}, \\
\Hscurl{s}{\Omega} &:= \Hspcurl{s}{2}{\Omega}.
\end{align*}
Furthermore we will use the notation
$$
H^{s,p}_0(\operatorname{curl};\Omega) := \{ U \in \Hspcurl{s}{p}{\Omega} \,:\, \nu \times U |_{\p\Omega} = 0 \},
$$
where we assume that $s \geq0$. We also define the space $\Hscurl{-1}{\Omega}$, as the dual
space of 
$$
\Hscurl{-1}{\Omega} := 
\big(H_0(\operatorname{curl};\Omega)\big)^*
$$
Next we will consider the traces of these spaces.

\subsection{Function spaces on $\p\Omega$.}
We now  consider the tangential trace spaces related to $\Hscurl{s}{\Omega}$.
Here one needs to be pay attention to the regularity of the boundary.
First assume that $\Omega \subset \R^3$ is a bounded domain with $C^\infty$-boundary.
The tangential trace spaces of  $\Hscurl{s}{\Omega}$ are
spaces of the form
\begin{align*} 
TH_D^{s}(\p\Omega)  
:= \{ F \in  H^{s} (\partial \Omega)^3 \,:\,  
\nu \cdot F|_{\p\Omega} = 0,
\,\nabla_{\p} \cdot F \in H^{s} (\partial \Omega)^3  \},  
\end{align*}
here $s \in \R$,
and  $\nabla_\p \,\cdot\,$ is the surface divergence (for more details see e.g. \cite{C}).
We  equip these spaces with the norm
$$
\| F \|_{ TH_D^{s}(\p\Omega)} := \big(\| F \|^2_{ H^{s}(\p\Omega)^3 } 
+ \| \nabla_\p \cdot F \|^2_{ H^{s}(\p\Omega)^3  }\big)^{1/2}.
$$
Likewise we can define the space
\begin{align*} 
TH_C^{s}(\p\Omega)  
:= \{ F \in  H^{s} (\partial \Omega)^3 \,:\,  
\nu \cdot F|_{\p\Omega} = 0,
\,\nabla_{\p} \times F \in H^{s} (\partial \Omega)^3  \},  
\end{align*}
$s\in \R$,
with the analogous norm to $TH_D^{s}(\p\Omega)$. Here $\nabla_{\p}\times$ is the surface curl.

In this paper we work with a $\p \Omega$ that is Lipschitz. The most 
natural smoothness index is in this case $s =-1/2$, and we thus work
with the spaces $TH_D^{-1/2}(\p\Omega)$ and $TH_C^{-1/2}(\p\Omega)$.
The following Lemma gives the basic result for the traces relating to $\Hcurl{\Omega}$.
For a proof in the case of a regular boundary see \cite{C} Theorem 4, p. 35 and Remark 5, p.36,
and for the case of a Lipschitz boundary see \cite{BCS} .
(The case of $s=1/2$, i.e. of more regular boundary data
on a Lipschitz regular boundary
seems to be more convoluted. Some results on $TH_D^{1/2}(\p \Omega)$ for Lipschitz boundaries
and Lipschitz polyhedra appear in \cite{BC,Bu}.)
\begin{lem} \label{lem_traceH} 
Let $\Omega \subset \R^3$ be a bounded domain with a Lipschitz boundary. Then the tangential traces
\begin{align*} 
&(\nu \times \,\cdot\,)|_{\p\Omega}:\Hcurl{\Omega}\to TH_D^{-1/2}(\p\Omega),\\
&(\nu \times (\nu \times \,\cdot\,))|_{\p\Omega}:\Hcurl{\Omega}\to TH_C^{-1/2}(\p\Omega),
\end{align*}
are continuous. Moreover we have the continuous right inverses
\begin{align*} 
&\eta_t : TH_D^{-1/2}(\p\Omega) \to \Hcurl{\Omega}, \\
&\eta_T : TH_C^{-1/2}(\p\Omega) \to \Hcurl{\Omega},
\end{align*}
such that $(\nu\times \eta_t f)|_{\p\Omega} = f$, and $\nu\times (\nu\times \eta_T f)|_{\p\Omega} = f$.
\end{lem}

\noindent
Note also that Lemma \ref{lem_traceH}, holds also in more regular spaces, so that 
e.g. $(\nu \times \,\cdot\,)|_{\p\Omega}:\Hscurl{s}{\Omega}\to TH_D^{1/2}(\p\Omega)$ is
continuous and has a continuous right inverse (see \cite{C}), when the boundary is 
sufficiently regular.

A somewhat surprising fact is that the space $TH_D^{-1/2}(\p\Omega)$ and the space
$TH_C^{-1/2}(\p\Omega)$ are dual spaces to each other
(see \cite{C} p. 38 Proposition 3 and Theorem 5.26 in \cite{KH} ). 

Another important fact is the integration by parts formula or Green's formula 
for curls.
For $U,V \in \Hcurl{\Omega}$, and a Lipschitz domain $\Omega \subset \R^3$, 
we have the integration by parts formula
\begin{align} \label{eq_intByParts}
\int_{\Omega} (\nabla \times U) \cdot V  \,dx
= \int_{\Omega} U \cdot (\nabla \times V)  \,dx
+ \langle \nu\times U , V \rangle_{\mathcal{D}(\p \Omega)},
\end{align}
where $\langle  \,\cdot\, , \,\cdot\,\rangle_{\mathcal{D}(\p\Omega)}$ stands for
the duality pairing between
$TH_D^{-1/2}(\p\Omega)$ and  $TH_C^{-1/2}(\p\Omega)$
(see Theorem 3.31 in \cite{Mo} and \cite{BCS}).
%

\medskip
\noindent
We will also need to consider more generally the traces of vector fields in $\Hspcurl{s}{p}{\Omega}$,
when $-1/p < s < 1 -1/p$ and $1<p<\infty$. Note  that the smoothness index $s\geq 1$
is not allowed, so that the space $\Hscurl{1}{\Omega}$ is not included.

The trace spaces of $\Hspcurl{s}{p}{\Omega}$ are generally
Besov type spaces,
which we define by
\begin{align*} 
TB_{p,p}^{s}(\p\Omega)  
:= \{ F \in  B^s_{p,p} (\partial \Omega)^3 \,:\,  
\nu \cdot F|_{\p\Omega} = 0,
\,\nabla_{\p} \cdot F \in B^{s}_{p,p} (\partial \Omega)^3  \},  
\end{align*}
$1<p<\infty$   and $-1 < s < 0$ (see Lemma \ref{lem_trace} below and furthermore \cite{M2}).
The norm of these spaces is given by
$$
\| F \|_{ TB_{p,p}^{-1/p}(\p\Omega)} := \big(\| F \|_{ B_{p,p}^{s}(\p\Omega)^3 } 
+ \| \nabla_\p \cdot F \|_{ B_{p,p}^{s}(\p\Omega)^3  } \big)^{1/p}.
$$
We will use the abbreviation
$$
TB_{D}^{-1/p}(\p\Omega)  := TB^{-1/p}_{p,p}(\p\Omega).
$$
These definition are close to the ones in \cite{M2}, and are motivated by Theorem 3.6 in \cite{M2}.
Note also that $TH_D^{-1/2}(\p\Omega) = TB^{-1/2}_D(\p\Omega)$.

Moreover we have the following  result for the tangential trace,
which follows directly from the results in \cite{M2}.

\begin{lem} \label{lem_trace}
Let $\Omega \subset \R^3$ be a bounded Lipschitz domain. Then the tangential trace 
$$
(\nu \times \,\cdot\,)|_{\p\Omega}:\Hspcurl{s}{p}{\Omega}\to TB^{s-1/p}_{p,p}(\p\Omega),
$$
is continuous for $1<p<\infty$ and $-1+1/p < s < 1/p$.
Moreover there exists a continuous right inverse
$$
\eta_{t,p} : TB^{s-1/p}_{p,p}(\p\Omega) \to \Hspcurl{s}{p}{\Omega},
$$
such that  $(\nu \times \eta_{t,p}f)|_{\p\Omega} = f$.
\end{lem}

\begin{proof}
Following \cite{M2} we define\footnote{Note that the smoothness index $s$ is interpreted differently 
in the spaces $TH^s(\p\Omega)$ and $\mathcal{TH}^{p}_{s}(\p\Omega)$.}
$$
\mathcal{TH}^{p}_{s}(\p\Omega) := \{ f \in B^{s-1/p}_{p,p}(\p\Omega) \,:\,
\exists U \in \Hspcurl{s}{p}{\Omega}, \text{ s.t. } \nu\times U|_{\p\Omega} = f \},
$$
and give it the norm 
$$
\| f\|_{ \mathcal{TH}^p_s(\p\Omega)} := 
\inf \{ \| U \|_{\Hspcurl{s}{p}{\Omega} } \,:\, \nu \times U |_{\p\Omega} = f\}.
$$
The tangential trace $(\nu \times \,\cdot\,)|_{\p\Omega}$ is hence an 
(topological and algebraic) isomorphism from
$$
(\nu \times \,\cdot\,)|_{\p\Omega} : 
\Hspcurl{s}{p}{\Omega}/H_0^{s,p}(\operatorname{curl};\,\Omega) 
\to \mathcal{TH}^{p}_{s}(\p\Omega). 
$$
By Theorem 3.6 in \cite{M2} 
$$
\| \, \cdot \,\|_{ \mathcal{TH}^{p}_{s}(\p\Omega) } \approx \| \,\cdot\, \|_{ TB^{s-1/p}_{p,p}(\p\Omega) },
$$
when $1<p<\infty$ and $-1+1/p < s < 1/p$.
Thus we have an isomorphism
$$
(\nu \times \,\cdot\,)|_{\p\Omega} : 
\Hspcurl{s}{p}{\Omega}/H_0^{s,p}(\operatorname{curl};\,\Omega) 
\to TB^{s-1/p}_{p,p}(\p\Omega), 
$$
which is continuous and has a continuous inverse.

\end{proof}

\medskip
\noindent
\textbf{Acknowledgments.} The author would like to thank Pedro Caro and Luca Rondi for 
some useful discussions. The author is supported by the grant PGC2018-094528-B-I00.

\end{document}